\documentclass[11pt, letterpaper]{article}    	
 \usepackage[margin=1in]{geometry}
\usepackage{setspace}
\usepackage{graphicx}	
\usepackage{subcaption}
\usepackage{float} 
\usepackage[toc]{appendix}

\usepackage[draft]{todonotes}   
\providecommand{\keywords}[1]
{
  \small	
  \textbf{\textit{Keywords---}} #1
}

\usepackage{enumerate}	
\usepackage{paralist}		

\usepackage{amssymb}
\usepackage{amsmath, amsthm, amssymb}

\newtheorem{thm}{Theorem}[section]
\newtheorem{conj}[thm]{Conjecture}
\newtheorem{prop}[thm]{Proposition}
\newtheorem{claim}{Claim}
\newtheorem{lemma}[thm]{Lemma}
\newtheorem{cor}[thm]{Corollary}

\theoremstyle{definition}

\newtheorem{defi}[thm]{Definition}

\def\F{\mathcal{F}}
\def\G{\mathcal{G}}
\def\HH{\mathcal{H}}
\def\A{\mathcal{A}}
\def\B{\mathcal{B}}
\def\I{\mathcal{I}}
\def\L{\mathcal{L}}
\def\qq{q}

\usepackage{tikz}
\newcounter{casenum}

\newcommand*{\rom}[1]{\expandafter{\romannumeral #1\relax}}

\usepackage{authblk}

\title{Intersecting families of sets are typically trivial}

\author{ J\'ozsef Balogh\thanks{Department of Mathematical Sciences, University of Illinois at Urbana-Champaign, IL, USA. Email: \texttt{jobal@illinois.edu}.
Research supported by NSF RTG Grant DMS-1937241, NSF Grant DMS-1764123 and Arnold O. Beckman Research Award (UIUC) Campus Research Board 18132, the Langan Scholar Fund (UIUC)  and the Simons Fellowship.}
\quad Ramon I. Garcia\thanks{Department of Mathematics, University of Illinois at Urbana-Champaign, Urbana, IL, USA. Email: \texttt{rig2@illinois.edu}.}
\quad Lina Li\thanks{Department of Combinatorics \& Optimization, University of Waterloo, Waterloo, Canada. Email: \texttt{lina.li@uwaterloo.ca}.}
\quad Adam Zsolt Wagner\thanks{Department of Mathematical Sciences, Worcester Polytechnic Institute, Worcester, USA. Email: \texttt{zadam@wpi.edu}.
}
}

\begin{document}

\maketitle
\begin{abstract}
A family of subsets of $[n]$ is \emph{intersecting} if every pair of its sets intersects. Determining the structure of large intersecting families is a central problem in extremal combinatorics. Frankl--Kupavskii and Balogh--Das--Liu--Sharifzadeh--Tran showed that for $n\geq 2k + c\sqrt{k\ln k}$, almost all $k$-uniform intersecting families are stars. Improving their result, we show that the same conclusion holds for $n\geq 2k+ 100\ln k$. Our proof uses, among others, the graph container method and the Das--Tran removal lemma.
\end{abstract}

\keywords{Extremal combinatorics,  Intersecting family, Graph container method}

\section{Introduction}

Several problems in extremal combinatorics are about determining the size or the structure of a system or collection of finite objects if it is known to satisfy certain restrictions. Once the answer to this extremal question is known, one can strengthen this result by enumerating such systems and determining their typical structure.
Some cornerstone results of this type are theorems of Kleitman~\cite{kleitmanantichain} who determined the log-asymptotics of the number of antichains in $\{0,1\}^n$, and of Erd\H{o}s--Kleitman--Rothschild~\cite{erdoskleitmanroth} who proved that almost all triangle-free graphs are bipartite. 
Motivated by these results, a large number of classical theorems in extremal combinatorics have been extended to enumerative and structural results in the past decades. In particular, the celebrated \emph{container method} of Balogh--Morris--Samotij~\cite{balogh2015independent} and Saxton--Thomason~\cite{saxton2015hypergraph} has seen particular success with such problems.


The main topic of this paper is the structure of \emph{intersecting} families. A family $\F\subset 2^{[n]}$ is \emph{intersecting} if every pair of members of $\F$ has a non-empty intersection. The seminal result of Erd\H{o}s--Ko--Rado~\cite{ekr} states that if $\F$ is a $k$-uniform intersecting family where $n\geq 2k$ then $|\F|\leq \binom{n-1}{k-1}$. The \emph{trivial intersecting family}, or star,  the family of all the sets that contain a fixed element, shows that this inequality is the best possible.

Improving a result of Balogh, Das, Delcourt, Liu and Sharifzadeh~\cite{balogh2015intersecting}, Frankl and Kupavskii~\cite{frankl2018counting} and independently Balogh, Das, Liu, Sharifzadeh and Tran~\cite{balogh2018structure} showed that if $n\geq 2k +2+c\sqrt{k \ln k}$, for some positive constant $c$, then almost all intersecting families are trivial. Let $I(n,k)$ denote the number of $k$-uniform intersecting families in $2^{[n]}$ and $I(n,k,\geq 1) $ the number of non-trivial such families.


\begin{thm}[Balogh--Das--Liu--Sharifzadeh--Tran~\cite{balogh2018structure}, Frankl--Kupavskii~\cite{frankl2018counting}]\label{thm:franklkupa}
For $n\geq  2k +2+c\sqrt{k \ln k}$ and $k \rightarrow\infty $ we have
$$I(n,k) = (n+o(1))2^{\binom{n-1}{k-1}} \quad \text{ and }\quad I(n,k,\geq 1) = o\left(2^{\binom{n-1}{k-1}}\right),$$
where $c=2$ in \cite{frankl2018counting} and $c$ was a large constant in \cite{balogh2018structure}. 
\end{thm}

A slightly weaker result, that covers the entire range of parameters, was obtained via the basic graph container method in~\cite{balogh2015intersecting}.

\begin{thm}[Balogh--Das--Delcourt--Liu--Sharifzadeh~\cite{balogh2015intersecting}]
For $n\geq 2k+1$ we have 




$$I(n,k) = 2^{(1+o(1))\binom{n-1}{k-1}},$$
where the $o(1)$ term goes to zero as $k\rightarrow\infty$.
\end{thm}

\noindent Our main result significantly strengthens the bound of Theorem~\ref{thm:franklkupa}.

\begin{thm}\label{thm:mainthm}
For $n\geq  2k + 100\ln k$, almost all intersecting families in $\binom{[n]}{k}$ are trivial. In particular, the number of non-trivial intersecting families $I(n,k,\geq 1)$ is $o\left(2^{\binom{n-1}{k-1}}\right)$ as $k \rightarrow\infty$.


\end{thm}

\noindent Here ``almost all'' means that the proportion of intersecting families that are not trivial tends to zero as $k$ increases. We believe the same conclusion holds whenever $n\ge 2k+2$. 

\begin{conj}\label{ConjMainThm}

For $n\geq  2k + 2$ and $k \rightarrow\infty $ we have
$$I(n,k,\geq 1) = o\left(I(n,k)\right).$$
\end{conj}

The other two remaining cases are $n\in\{2k,2k+1\}$. When $n=2k$, we can choose at most one set from every complementary pair so that $$I(2k,k)=3^{\frac{1}{2}\binom{2k}{k}}.$$

When $n=2k+1$, it is not true that almost all intersecting families are trivial, as the following construction from \cite{balogh2018structure} shows. 
For an intersecting family $\A\subset \binom{[2k]}{k}$, let
$$\F_{\A}:=\left\{F\in\binom{[2k+1]}{k}: 2k+1 \in F, ~ F\cap A \neq \emptyset \text{ for }\forall A\in\A\right\} \cup \A.$$
Note that if $\F'$ is a subfamily  of  some $\F_{\A}$ such that $\A\subset\F'$, then we can uniquely recover $\A$. 
By considering only families $\A\subset \binom{[2k]}{k}$ of size $|\A|=\binom{2k}{k}2^{-k}$, of which almost all are intersecting, a similar calculation as in \cite[Example~3.1]{balogh2021independent} shows that there are at least 
\[
\begin{split}
\sum_{\A} 2^{|\L_{k-1}| - |N(\A)|} 
&\geq
\exp\left( \binom{2k}{k}2^{-k}\right)
2^{k\binom{2k}{k}2^{-k}} \cdot 2^{\binom{2k}{k-1} - k\binom{2k}{k}2^{-k} + \frac{k^2}{2}\binom{2k}{k}2^{-2k}- o\left(k^{3/2}\right)}\\
&\geq
2^{\binom{2k}{k-1}}\exp\left(\binom{2k}{k}2^{-k} + \frac{k^2\ln 2}{2}\binom{2k}{k}2^{-2k} - o\left(k^{3/2}\right)\right)
\end{split}
\]
$k$-uniform intersecting families in $2^{[2k+1]}$, where we use the fact that the number of such families $\A$ is $(1-o(1))\exp\left( \binom{2k}{k}2^{-k}\right)
2^{k\binom{2k}{k}2^{-k}}$, and that almost all such $\A$ satisfy $|N(\A)|= k\binom{2k}{k}2^{-k} - \frac{k^2}{2}\binom{2k}{k}2^{-2k} + o\left(k^{3/2}\right)$.
This is significantly larger than the number $(2k+1-o(1))2^{\binom{2k}{k-1}}$ of trivial $k$-uniform intersecting families.

Let us describe what we believe the typical intersecting family looks like. For a family $\A\subset \binom{[2k]}{k}$ we can define an auxiliary graph $G_\A$ with vertex set $\A$, where for $X,Y\in \A$ we include the edge $XY$ if $|X\cap Y|=k-1$.
Given a set family $\F \subset \binom{[2k+1]}{k}$, we say that $\F$ is \textit{nice}, if there is an index (to simplify notation, assume $2k+1$), such that the following holds:
\begin{itemize}
    \item for $\A= \{F: 2k+1\not\in F\in \F\}$ the graph $G_\mathcal{A}$ has components only of sizes 1 or 2;
    \item $A\cap B\neq \emptyset$ for all $A\in \A$ and $B\in \F\setminus\A$.
\end{itemize}
We remark that a nice family is not necessarily intersecting, as such $\A$ might not be intersecting.
By constructing the \textit{abstract polymer model} on connected subsets of $G_{\binom{[2k]}{k}}$ of sizes only 1 or 2, and calculating its \textit{cluster expansion} with the weight function $w(S)=2^{-|N(S)|}$, one can show that the number of nice families is at most 
\[
n2^{\binom{2k}{k-1}} \exp\left(\binom{2k}{k}2^{-k} + \frac{k^2-1}{2}\binom{2k}{k}2^{-2k} + o(1)\right).
\]
For detailed explanation of the cluster expansion and its application, we refer interested readers to \cite{balogh2021independent, jenssen2020independent}.
\begin{conj}
Almost all intersecting families $\F\subset \binom{[2k+1]}{k}$ are nice. In fact, the number of intersecting families is 
$$2^{\binom{2k}{k-1}} \exp\left(\binom{2k}{k}2^{-k} + \Theta\left(k^{3/2}\right)\right).$$
\end{conj}

The paper is structured as follows. In the remaining of this section we introduce some of the notation we use. In Section~\ref{sec:prelim} we associate each non-trivial $k$-uniform intersecting family to an independent set in a fixed biregular graph; we also state the main tools we use, including the graph container lemma we need. The proof of the graph container lemma is deferred to Section~\ref{sec:contlemproof}. 
We split the proof of Theorem~\ref{thm:mainthm} into two cases, according to the sizes of the $2$-linked components in the associated independent set; Section~\ref{sec:bigcomp}  and Section~\ref{sec:smallcomp} each corresponds to dealing with one of the cases.

\subsection{Notation}
Let $G=(V, E)$ be a simple graph.
The \textit{neighborhood} of a set $A$ of vertices in $G$, denoted by $N_G(A)$, is the set of vertices adjacent to $A$ in $G$.
For a set $S\subseteq V$, the \textit{neighborhood of $A$ restricted to $S$}, denoted by $N_{S, G}(v)$, is the set of vertices adjacent to $A$, which are contained in $S$.
The \textit{distance} $d_G(u, v)$ between two vertices $u, v$ is the length of the shortest path between $u$ and $v$ in $G$.
When the underlying graph is clear, we simply write $N(A)$, $N_S(A)$ and $d(u, v)$ instead.

Let $\Sigma$ be a bipartite graph with classes $X$ and $Y$.
For $A\subseteq X$, denote by $[A]:=\{v\in X: N(v)\subseteq N(A)\}$ the \textit{closure} of $A$ and call $A$ \emph{closed} if $[A]=A$.
We say a pair of vertices $u$ and $v$ is \textit{$k$-linked} in a set $A\subseteq X$, if there exists a sequence $u=v_1, v_2, \ldots, v_{\ell-1}, v_{\ell}=v$ in $A$ such that for each $i\in [\ell-1]$, the distance $d(v_{i}, v_{i+1})$ is at most $k$. 
A set $A$ is \textit{$k$-linked} if every pair of vertices in $A$ is $k$-linked in $A$.
A \textit{$k$-linked component} of a set $B\subseteq X$ is a maximal $k$-linked subset of $B$.

All instances of $\log$ refer to logarithm of base $2$. 
We ignore all floors and ceilings whenever  they are not crucial.
We also use the notation $\binom{n}{\leq k}$ as  a shorthand for $\sum_{0\leq i\leq k}\binom{n}{i}$. 
 
\section{Preliminary}\label{sec:prelim}
Let $n=2k+r$. By Theorem~\ref{thm:franklkupa}, without
specification, we may assume from now on that $k$ is sufficiently large and 
\begin{equation}\label{rrange}
100\ln k \leq r \leq 2 + 2\sqrt{k\ln k}.
\end{equation}
\subsection{Relation to independent sets}\label{subsec:relation}
Let $\mathcal{F}$ be an intersecting family in $\binom{[n]}{k}$. Define $f(\F)$ to be the most frequent element in the sets of $\F$; in case there are multiple choices, we always choose the one with the largest label.
The star of $\mathcal{F}$, denoted by $\mathcal{S}_{\mathcal{F}}$, is the set of all $k$-element subsets of $[n]$ that contain $f(\F)$.

An intersecting family is said to be \textit{maximal}, if no additional set can be added without destroying the intersecting property. We fix a linear ordering $\prec$ on the set of intersecting families in $\binom{[n]}{k}$. Let $\F^*$ be the maximal intersecting family that contains $\F$ and minimizes\footnote{Indeed, the particular choice of maximal family $\F^*$ has no significance; we simply need a way to fix the selection.} $|\F^*\Delta \mathcal{S}_{\F^*}|$; in case there are multiple such families, we choose the first one under the ordering $\prec$. 
For ease of notation, assume that $f(\mathcal{F^*})=n$. Define 
\begin{equation}\label{eq:defA}
\mathcal{A}:=\{[n-1]\setminus A: n\notin A\in \mathcal{F}\} \subseteq \binom{[n-1]}{k+r-1}
\end{equation}
and
\begin{equation}\label{eq:defB}
\mathcal{B}:=\{B\setminus\{n\}: n\in B\in \mathcal{F}\} 
\subseteq \binom{[n-1]}{k-1}.
\end{equation}
It is not hard to check that $A$ is also an intersecting family in $\binom{[n-1]}{k+r-1}$ for every $r \geq 1$.
Let $\HH$ be the bipartite graph defined on 
\[
\mathcal{L}_{k-1}:=\binom{[n-1]}{k-1} \quad \text{ and } \quad \mathcal{L}_{k+r-1}:=\binom{[n-1]}{k+r-1},
\]
such that two sets are adjacent if and only if one contains the other. 
Observe that $\mathcal{A}\cup \mathcal{B}$ is an independent set in $\mathcal{H}$. Indeed, otherwise, there exist sets $A, B\in\mathcal{F}$, where $n\in B$ and $n\notin A$, such that $B\setminus\{n\} \subseteq [n-1]\setminus A$, which implies that $A\cap B=\emptyset$, i.e., contradicts that $\F$ is an intersecting family.

Roughly speaking, this indicates that each intersecting family in $\binom{[n]}{k}$ is corresponding to a unique independent set $I$ in $\HH$ (but not vice versa, as $I\cap \binom{[n-1]}{k+r-1}$ is not necessarily intersecting). 
More precisely, let $\I(n,k)$ be the family of $k$-uniform intersecting families in $2^{[n]}$, and $\I(\HH)$ be the family of independent sets in $\HH$.
Then we can define an injective map $\phi: \I(n, k) \rightarrow [n] \times \I(\HH)$ such that \[
\phi(\F):=(f(\F^*), \A\cup \B),\] where $\A\subseteq \L_{k+r-1}$ and $\B\subseteq \L_{k-1}$.
Whenever $f(\F^*)$ is not crucial, we simply refer $\phi(\F)$ to $\A\cup \B$.
Given the close relation between intersecting families and independent sets of $\HH$, we can always learn about one object through the properties of the other one.

\subsection{Tools}
Recall that a \textit{composition} of an integer $n$ is an ordered sequence $\langle a_1,\ldots\rangle $ of positive integers summing to $n$, the $a_i$’s are the \textit{part}s of the composition.

\begin{prop}\label{prop:decom}
The  number  of  compositions  of $n$ is $2^{n-1}$ and  the  number of compositions with  at  most $b$ parts is $\sum_{i<b}\binom{n-1}{i}<2^{b\log(en/b)}$, when $b< n/2$.
\end{prop}

We also need the following lemma from~\cite{Galvin2004Phase}, which bounds the number of connected subsets of a graph.
\begin{lemma}\label{lem:connected_count}
For a graph $G$ with maximum degree $\Delta$, the number of $\ell$-vertex subsets of $V(G)$ which contains a fixed vertex and induce a connected subgraph is at most $(e\Delta)^{\ell}$.
\end{lemma}
Given an $(r, s)$-biregular graph $G$, we can create an auxiliary graph $G_k$ on the same vertex set by connecting two vertices if their distance is at most $2k$. Then $G_k$ has maximum degree at most $r^ks^k$. Applying Lemma~\ref{lem:connected_count} to one part of $G_k$, we get the following corollary.
\begin{cor}\label{cor:numlinkset}
For $r, s\geq 2$, let $\Sigma$ be an $(r, s)$-biregular graph with bipartition $X\cup Y$. Then for every $m\geq 1$, the number of $2m$-linked subsets of $X$ of size $\ell$ which contain a fixed vertex is at most $\exp(2\ell m\ln(rs))$.
\end{cor}

Next we will present some useful properties of $\HH$.
First observe that $\mathcal{H}$ is a $\left(\binom{k+r-1}{k-1}, \binom{k+r}{r}\right)$-regular bipartite graph with the bipartition $\L_{k+r-1}\cup \L_{k-1}$. Define 
\[
d:=\binom{k+r-1}{k-1}.
\]
Then $\mathcal{H}$ is `almost' $d$-regular as $\binom{k+r}{r}=\frac{k+r}{k}\binom{k+r-1}{k-1}=\frac{k+r}{k}d$ and $r=o(k)$. Notice that
\begin{equation}\label{eq:lowerBoundd}
d=\binom{k+r-1}{k-1}=\frac{(k+r-1)\cdot \ldots \cdot k}{r!}\geq \left(\frac{k}{r}\right)^r    
\end{equation}
and
\begin{equation}\label{eq:upperboundd}
    d=\frac{(k+r-1)\cdot \ldots \cdot k}{r!}\leq\frac{(2k)^r}{(r/5)^r}=\left(\frac{10k}{r}\right)^r, 
\end{equation}
where the last inequality is obtained using Stirling's approximation.

We use isoperimetric inequalities on $\HH$, which can be easily derived from
direct applications of the Lov{\' a}sz version of the Kruskal-Katona Theorem~\cite{kruskal1963number, katona2009theorem}.

\begin{thm}[Lov{\' a}sz~\cite{lovasz2007combinatorial}]\label{thm:lovasz}
Let $\mathcal{A}$ be a family of $m$-element subsets of a fixed set U and $\mathcal{B}$ be the family of all $(m-q)$-element subsets of the sets in $\mathcal{A}$. If $|\mathcal{A}|=\binom{x}{m}$ for some real number $x$, then
$|\mathcal{B}|\geq \binom{x}{m-q}$.
\end{thm}

\begin{thm}[Isoperimetry]\label{thm:iso}
Assume that $1\leq r\leq 2+ 2\sqrt{k\ln k}$. Let $\mathcal{A}\subseteq\mathcal{L}_{k+r-1}$.
\begin{itemize}

    \item[(i)] If $|\mathcal{A}|\leq \binom{n-2 -c}{k+r-1}$ for some positive integer $c$, then 
    \[
    |N(\mathcal{A})|\geq |\mathcal{A}|\left(1 + \frac{c}{k+r-c-1}\right)\cdot\ldots \cdot\left(1 + \frac{c}{k-c}\right).
    \]
    \item[(ii)] If $|\mathcal{A}|\leq d^3$, then $|N(\A)|\geq d|\A|/(5e)^r$.
\end{itemize}
\end{thm}

\begin{proof}
(i) Let $|\mathcal{A}|=\binom{x}{k+r-1}$ for some $x$. Using the upper bound on the size of $\mathcal{A}$, we have $x\leq 2k
+r-2 - c$.
By Theorem~\ref{thm:lovasz}, we have
\[
\begin{split}
|N(\mathcal{\mathcal{A}})|&\geq \binom{x}{k-1}=\binom{x}{k+r-1}\frac{(k+r-1)\cdot\ldots\cdot k}{(x-k+1)\cdot\ldots \cdot(x-k-r+2)}\\
&\geq |\mathcal{A}|\frac{(k+r-1)\cdot\ldots\cdot  k}{(k+r-c-1)\cdot\ldots \cdot(k-c)}
=|\mathcal{A}|\left(1 + \frac{c}{k+r-c-1}\right)\cdot\ldots \cdot\left(1 + \frac{c}{k-c}\right).
\end{split}
\]

(ii) Let $|\mathcal{A}|=\binom{x}{k+r-1}$ for some $x$. 
First note that $(4r)!\leq 4^{4r}(r!)^4$.
By the assumption on the size of $|\mathcal{A}|$, we have $x\leq k+5r-1$, as otherwise
\[
\begin{split}
\binom{x}{k+r-1}&>\binom{k+5r-1}{k+r-1}=\frac{(k+5r-1)\cdot\ldots\cdot (k+r)}{(4r)!} \\
&> \frac{k^r}{4^{4r}r!}\left(\frac{(k+r-1)\cdot\ldots\cdot k}{r!}\right)^3 
\geq \binom{k+r-1}{r}^3=d^3,
\end{split}
\]
where the last inequality follows from the upper bound of $r$.
Similarly as in (i), by Theorem~\ref{thm:lovasz}, we have
\[
\begin{split}
|N(\mathcal{A})|&\geq \binom{x}{k-1}=\binom{x}{k+r-1}\frac{(k+r-1)\cdot\ldots\cdot  k}{(x-k+1)\cdot \ldots \cdot(x-k-r+2)}\\
&\geq |\mathcal{A}|\frac{(k+r-1)\cdot\ldots\cdot k}{(5r)\cdot\ldots\cdot  (4r+1)}
=d|\A|/\binom{5r}{r} \geq \frac{d|\mathcal{A}|}{(5e)^r}.
\end{split}
\]
\end{proof}

We state and use a special case of the Das--Tran  \cite[Theorem 1.2]{DasTranRemoval} removal lemma, obtained by setting $\ell=1$ and $\beta=0$, in their notation. 
\begin{lemma}[Das and Tran~\cite{DasTranRemoval}]\label{lem: D&T}
There is an absolute constant $C>1$ such that if $n$ and $k$ are positive integers satisfying $n>2k$, and $\mathcal{F}\subset \binom{[n]}{k}$ is an intersecting family of size $|\mathcal{F}|=(1-\alpha)\binom{n-1}{k-1}$, where $2 |\alpha|\leq \frac{n-2k}{(20C)^2n}$, then there exists a family $\mathcal{S}\subset \binom{[n]}{k}$ that is a star, satisfying 
\[
|\mathcal{F}\Delta \mathcal{S}|\leq C\alpha \frac{n}{n-2k}\binom{n-1}{k-1}.
\]
\end{lemma}

\subsection{The graph container theorem}
The most important tool of this paper is the following graph container theorem, which estimates the number of 2-linked sets in $\HH$. The proof follows from a graph container lemma of Sapozhenko~\cite{sapozhenko1987number} and will be postponed to Section~\ref{sec:contlemproof}.

\begin{thm}[Graph container theorem]\label{containerlemmalayer}
Assume that $1\leq r \leq 2+2\sqrt{k\ln k}$.
Recall that $\HH$ is a $(d, \frac{k+r}{k}d)$-biregular bipartite graph on $\L_{k+r-1}\cup \L_{k-1}$, where $d=\binom{k+r-1}{k-1}$.
For integers $a, g\geq 1$, let
\[
\mathcal{G}(a, g)=\{A\subseteq \mathcal{L}_{k+r-1}: A \text{ is a } 2\text{-linked set},\ |[A]|=a,\ |N(A)|=g\},
\]
and 
\[t:=\frac{k+r}{k}gd -ad.\]
For every $1\leq a \leq \binom{2k+3r/4}{k+r-1}$ and $1\leq \varphi, \psi \leq d -1$, there exists a family of containers $\mathcal{W}(a, g)\subseteq 2^{\mathcal{L}_{k+r-1}}\times 2^{\mathcal{L}_{k-1}}$ of size
\[
\begin{split}
|\mathcal{W}(a, g)|
\leq&\ |\mathcal{L}_{k-1}|\exp\left(
O\left(\frac{g\ln^2 d}{\varphi d}\right) 
+ O\left(\frac{t\ln^2 d}{d(d - \varphi)}\right)
+ O\left(\frac{t\ln^2 d}{\varphi d}\right)\right)\cdot\\
&
\exp\left( O\left(\frac{t\ln d}{(d - \varphi)\psi}\right)
+ O\left(\frac{t\ln d}{(d - \psi)\psi}\right)
\right),
\end{split}
\] and a function $f: \mathcal{G}(a, g) \rightarrow \mathcal{W}(a, g)$ such that for each $A\in \mathcal{G}(a, g)$, the pair $(S, F):=f(A)$ satisfies:
\begin{itemize}
    \item[(i)] $S\supseteq [A]$, $F\subseteq N(A)$;
    \item[(ii)] $|S| \leq \frac{k+r}{k}|F|  + \frac{\psi t}{d}\left(\frac{1}{d -\psi} + \frac{1}{\frac{k+r}{k}d -\psi}\right)$.
\end{itemize}
\end{thm}

\subsection{Outline of the proof of Theorem 1.3}\label{subsec:thm1.3sketch}


First we will assign to each non-trivial intersecting family $\mathcal{F}\subset\binom{[n]}{k}$ an independent set $\phi(\mathcal{F})$ in $\HH$ through the process described in Section~\ref{subsec:relation}. After doing this, we split the counting into cases according to whether $|[\phi(\mathcal{F})\cap \mathcal{L}_{k+r-1}]|\geq \binom{2k+3r/4}{k+r-1}$ or not.

In Section \ref{sec:bigcomp}, we prove that the number of non-trivial intersecting families with $|[\phi(\mathcal{F})\cap \mathcal{L}_{k+r-1}]|\geq \binom{2k+3r/4)}{k+r-1}$ is $2^{o(\binom{k-1}{n-1})}$. The proof, which is similar to that of~\cite[Theorem~1.6]{balogh2018structure}, is based on the Das-Tran \cite{DasTranRemoval} Removal Lemma, and an application of the Bollob\'as set pairs inequality on maximal intersecting families, see~\cite[Proposition 2.2]{balogh2015intersecting}.

We will dedicate Section~\ref{sec:smallcomp} to bound the number independent sets $I$ with $|[I\cap \mathcal{L}_{k+r-1}]|< \binom{2k+3r/4}{k+r-1}$. 
Like many of the advances in counting independent sets (for example, see~\cite{balogh2021independent, galvin2011threshold, jenssen2020independent, sapozhenko1987number}), we will use the graph container method of Sapozhenko~\cite{sapozhenko1987number}, together with some isoperimetric properties of $\HH$ (see Theorem~\ref{thm:iso}).

Adding up the bounds obtained from Sections \ref{sec:bigcomp} and \ref{sec:smallcomp} will give an upper bound for the number of non-trivial intersecting families $\mathcal{F}\subset \binom{[n]}{k}$, which leads to Theorem \ref{thm:mainthm}.

\section{Intersecting families with large components}\label{sec:bigcomp}

In this section, we will prove the following lemma,
which helps us to narrow down the family of intersecting families we have to care about. 
The proof uses an idea from \cite[Theorem 1.6]{balogh2018structure}.

Recall from Section~\ref{subsec:relation} that for each intersecting family $\F\subset \binom{[n]}{k}$, there exists a unique set $\phi(\F)$ such that $\phi(\F)$ is independent in $\HH$.

\begin{lemma}\label{claim:componentsaresmall}

The number of intersecting families $\F\subset\binom{[n]}{k}$ with
\begin{equation}\label{case:hughcomp}
\left|[\phi(\F)\cap \mathcal{L}_{k+r-1}]\right|> \binom{2k+ 3r/4}{k+r-1}
\end{equation}
is at most $2^{\binom{n-1}{k-1}-\frac{1}{2}n\binom{2k}{k}}$.
\end{lemma} 

\begin{proof}
Recall that an intersecting family is \textit{maximal}, if no additional set can be added without destroying the intersecting property.
For each $\ell\in \mathbb{N}$, let $M_{\ell}$ be the collection of maximal non-trivial $k$-uniform intersecting families of size exactly $\binom{n-1}{k-1}-\ell$.
Hilton and Milner \cite{HiltonMilnerStability} proved that if $n\geq 2k+1$ and $\mathcal{F}$ is a non-trivial $k$-uniform intersecting family, then $|\mathcal{F}|\leq \binom{n-1}{k-1}-\binom{n-k-1}{k-1}+1$. Therefore, we can assume that $\ell\ge \binom{n-k-1}{k-1}-1$, as otherwise $|M_{\ell}|=0$. 

Fix an arbitrary intersecting family $\F$ satisfying (\ref{case:hughcomp}). 
Let $\F^*$ be the maximal intersecting family defined in Section \ref{subsec:relation} with $\F\subseteq \F^*$ and minimum $|\F^*\Delta \mathcal{S}_{\F^*}|$.
We will show that 
\[
\F^*\in M_{\ell}\quad \text{for  some }\ell > n\binom{2k}{k}.
\]

Assume for contradiction that $\ell \leq n \binom{2k}{k}$. 
Set $\alpha:=\ell\cdot \binom{n-1}{k-1}^{-1}$. Then we have
\[
\begin{split}
\alpha
&\leq n\binom{2k}{k}\binom{n-1}{k-1}^{-1}=\frac{n(k+r)}{k}\frac{(k+r-1)\cdot\ldots\cdot (k+1)}{(2k+r-1)\cdot\ldots\cdot (2k+1)}
\leq \frac{n(k+r)}{k}\left(\frac{2}{3}\right)^{r-1}
\leq \frac{n-2k}{(20C)^2 n}
\end{split}
\]
for some constant $C$, where the last two inequalities follow from the range of $r$, see~(\ref{rrange}). Note that the last term in the above inequality is a requirement for Lemma~\ref{lem: D&T}.
Applying Lemma~\ref{lem: D&T} on $\F^*$ with $\alpha$, we obtain that

\[
 |\mathcal{F}^*\Delta\mathcal{S_{\mathcal{F}^*}}|\leq C\alpha\frac{n}{n-2k}\binom{n-1}{k-1} \leq Cn\ell \leq Cn^2\binom{2k}{k},
\]
where $\mathcal{S}_{\mathcal{F}^*}$ is the star consisting of all $k$-sets of $[n]$ containing the most frequent element of $\F^*$. 
Note by the definition of $\phi$ that $\phi(\F)\subseteq\phi(\F^*)$, and both of them are independent sets in the $(d, \frac{k+r}{k}\cdot d)$-regular graph $\HH$. 
It follows that
\begin{equation}\label{eq:ubound_closure}
\begin{split}
\left|[\phi(\F) \cap \L_{k+r-1}]\right|&
\leq \frac{k+r}{k}\left|N(\phi(\F)\cap \L_{k+r-1})\right|\leq \frac{k+r}{k}\left|N(\phi(\F^{*})\cap \L_{k+r-1})\right|\\&\leq 2\left|\L_{k-1} - \phi(\F^{*})\right| 
\leq 2\left|\mathcal{F^{*}}\Delta\mathcal{S}_{\mathcal{F}^*}\right| \leq 2Cn^2\binom{2k}{k}.
\end{split}
\end{equation}

\noindent For the lower bound, by the assumption~(\ref{case:hughcomp}), we have
\[
\begin{split}
\left|[\phi(\F)\cap \mathcal{L}_{k+r-1}]\right|
&> \binom{2k+ 3r/4}{k+r-1}
=\binom{2k}{k}\prod_{j=1}^{r-1}\frac{2k+j}{k+j}\prod_{j=\lfloor 3r/4\rfloor+1}^{r-1}\frac{k-r+1+j}{2k+j}\\
&\geq \binom{2k}{k}\left(\frac{2k+r}{k+r}\right)^{r-1}\left(\frac{k-r/4}{2k+3r/4}\right)^{r/4}.
\end{split}
\]


\noindent Since $r\leq 2 + 2\sqrt{k\ln k}$ and $k$ is sufficiently large, we have $1.9\leq (2k+r)/(k+r)$ and $0.4\leq (k-r/4)/(2k+3r/4)$. Then it follows that
\[
\left|[\phi(\F)\cap \mathcal{L}_{k+r-1}]\right|\geq \binom{2k}{k}1.9^{r-1}\cdot 0.4^{r/4}>\binom{2k}{k}1.5^{r}>  \binom{2k}{k}n^3,
\]
where the last inequality uses $r\geq 100 \ln k$. This contradicts (\ref{eq:ubound_closure}), and then shows $\ell> n\binom{2k}{k}$.

Therefore, each intersecting family with property~(\ref{case:hughcomp}) is contained in some $\mathcal{F}^*\in\cup_{\ell> n\binom{2k}{k}}M_{\ell}$. 
From~\cite[Proposition 2.2]{balogh2015intersecting}, we know that $\sum_{\ell} |M_{\ell}|\leq 2^{\frac{1}{2}n\binom{2k}{k}}$.
Hence, the number of such intersecting families is at most
\[
\sum_{\binom{n-1}{k-1}\ge\ell>n \binom{2k}{k}}|M_{\ell}|\cdot2^{\binom{n-1}{k-1}-\ell}\leq 2^{\binom{n-1}{k-1}}\sum_{\ell> n\binom{2k}{k}}|M_{\ell}|2^{-\ell}\leq 2^{\binom{n-1}{k-1}-\frac{1}{2}n\binom{2k}{k}},
\]
as desired.

\end{proof}

\section{Independent sets with small components}
The next theorem will be used to bound the number of families that remain to be counted by Lemma \ref{claim:componentsaresmall}. The proof uses ideas from~\cite{kahn2019number} and~\cite{sapozhenko1987number}.
\label{sec:smallcomp}
\begin{thm}\label{thm:indset}
Let $\mathcal{I}$ be the collection of independent sets $I$ in $\mathcal{H}$ with $I\cap \mathcal{L}_{k+r-1}\neq \emptyset$ and $|[I\cap \mathcal{L}_{k+r-1}]| \leq \binom{2k + 3r/4}{k+r-1}$. Then $|\mathcal{I}|\leq 2^{\binom{n-1}{k-1}-\frac{\sqrt{d}}{4k}}$.
\end{thm} 

Before we move to the technical proof of Theorem~\ref{thm:indset}, we first show how this theorem will be applied to complete the proof of Theorem \ref{thm:mainthm}.

\begin{proof}[Proof of Theorem \ref{thm:mainthm}]
First by Lemma~\ref{claim:componentsaresmall}, the number of intersecting families with property~(\ref{case:hughcomp}) is $o\left(2^{\binom{n-1}{k-1}}\right)$. For the rest of the non-trivial intersecting families $\F$, they all satisfy $1\leq |[\phi(\F) \cap \mathcal{L}_{k+r-1}]| \leq \binom{2k + 3r/4}{k+r-1}$. 
From here, we instead count the independent sets $\phi(\F)$, which by Theorem~\ref{thm:indset} is at most $n2^{\binom{n-1}{k-1}-\frac{\sqrt{d}}{4k}} = o\left(2^{\binom{n-1}{k-1}}\right)$,
where $n$ refers to the number of choices of the most frequent element.
In sum, we easily have $I(n, k, \geq 1) = o\left(2^{\binom{n-1}{k-1}}\right)$. Note that the number of trivial intersecting families is at least $2^{\binom{n-1}{k-1}}$, and therefore the typical structure naturally follows from the counting result.

\end{proof}

\begin{proof}[Proof of Theorem~\ref{thm:indset}]
Let $I$ be an independent set in $\mathcal{I}$. We first define 
\[
a:=\left|[I\cap \mathcal{L}_{k+r-1}]\right| \quad\text{and}\quad g:= \left|N([I\cap \mathcal{L}_{k+r-1}])|= |N(I\cap \mathcal{L}_{k+r-1})\right|.
\]
Let $\{A_1,A_2,\ldots, A_i, \ldots\}$ be the collection of $2$-linked components of $I\cap \mathcal{L}_{k+r-1}$. By the definition of $2$-linked component, all $N(A_i)$'s are pairwise disjoint. 
Note that the set $I\cap \mathcal{L}_{k+r-1}$ is not necessarily closed; it is closed when the family $\mathcal{F}$ associated to $I$ is maximal. 
We write
\[
a_i:=|[A_i]| \quad\text{and}\quad g_i:=|N([A_i])|=|N(A_i)|.
\]
Note that for every vertex $v$ in $[I\cap \mathcal{L}_{k+r-1}]$, there must be some unique $i$ such that $N(v)\subseteq N(A_i)$, as otherwise, it violates the fact that $A_i$ is a 2-linked component. Therefore, we have $a=\sum{a_i}$ and $g=\sum g_i$. 
Moreover, by the assumption of Theorem~\ref{thm:indset}, we have that for every  $i$,
\begin{equation}\label{abound}
1\leq a_i\leq a\leq \binom{2k + 3r/4}{k+r-1}.
\end{equation}

Recall that for each independent set $I$ in $\HH$, we denote by $\{A_1, A_2,\ldots\}$ the collection of its 
$2$-linked components in $\mathcal{L}_{k+r-1}$. 
Given the collection $\{A_1, A_2,\ldots\}$, define the set $C$ as 
\[C:=\left(\bigcup{A_i}\right)\cup\left(\mathcal{L}_{k-1} - \bigcup{N(A_i)}\right).
\]
Notice that $I\subseteq C$. Instead of counting independent sets
directly, we will estimate the number of possible sets $C$ and bound the number of independent sets assigned to each $C$. In particular, 
for a given set $C$ with 2-linked components $\{A_1,A_2,\ldots\}$, the number of independent sets assigned to $C$ is exactly 
\begin{equation}\label{eq:boundindsets}
2^{|\L_{k-1} - \bigcup N(A_i)|} = 2^{\binom{n-1}{k-1} - g},
\end{equation}
as $I\cap \L_{k+r-1}$ is fixed by the set $C$ and all vertices in $\L_{k-1}$ can appear in $I$ except for those adjacent to $A_i$'s. Let $C_{g}$ denote the number of sets $C\subset V(\HH)$ that are associated to an independent set $I$ with $|N(I\cap \mathcal{L}_{k+r-1})|=g$.   We have that the number of independent sets we would like to count, $|\mathcal{I}|$, is bounded by
\begin{equation}\label{eq:Cgappears}
|\mathcal{I}|\leq \sum_{g}C_g \cdot 2^{\binom{n-1}{k-1}- g}.
\end{equation}


For a pair of integers $(a_i, g_i)$ with $1\leq a_i < g_i$ and $a_i\leq \binom{2k+3r/4}{k+r-1}$, let
\[
\G(a_i, g_i):=\left\{A_i \mid A_i\text{ is 2-linked, } |[A_i]|=a_i \text{ and } |N(A_i)|= g_i\right\}.
\]
Recall that $\HH$ is a $\left(d, \frac{k+r}{k}d\right)$-regular bipartite graph where $d=\binom{k+r-1}{k-1}$. We may thus assume $g_i\geq d$ as otherwise $|\G(a_i, g_i)|$ is trivially zero.
From now on, we fix $a_i$ and $g_i$ and obtain a bound on $|\G(a_i,g_i)|$. For this, we split the proof into two cases.

~

\noindent\textbf{Small case.} $g_i \leq d^3$. 

\noindent By Corollary~\ref{cor:numlinkset}, we have at most 
\[
|\mathcal{L}_{k+r-1}|2^{2a_i\log(d\cdot d\frac{k+r}{k})}\leq 2^{n+5a_i\log d}
\]
options for $[A_i]$ and with it fixed, $2^{a_i}$ options for $A_i$. We have 
\begin{equation}\label{eq:smallcasebound}
\begin{split}
|\G(a_i,g_i)|\leq 2^{n+5a_i\log d+a_i}
&\leq 2^{n+6a_i\log d} 
\leq  2^{n+6(5e)^rg_i\log d/d}
\leq 2^{g_i/4 + g_i/4}=2^{g_i/2},
\end{split}
\end{equation}
where the second inequality follows from Theorem~\ref{thm:iso}(ii), and the last inequality follows from $g_i\geq d$ and $d=\binom{k+r-1}{r}\geq k^2/2$.

~

\noindent\textbf{Large case.} $g_i \geq d^3$.

\noindent In this case, we will prove that
\begin{equation}\label{eq:largecasebound}
\begin{split}
|\G(a_i,g_i)|\leq 2^{g_i\left(1 - \frac{1}{k\sqrt{d}}\right)}.
\end{split}
\end{equation}
Since $a_i\leq \binom{2k+3r/4}{k+r-1}$, applying Theorem~\ref{thm:iso}(i) with $c=r/4 -2$, we have 
\begin{equation}\label{eq:glower}
g_i\geq a_i \left(1 + \frac{r/4 -2}{k + 3r/4 +1} \right) \cdot \ldots \cdot \left(1+\frac{r/4 -2}{k - r/4+2}\right) 
\geq a_i\left(1 + \frac{r^2}{8k}\right).
\end{equation}
Let $t_i=\frac{k+r}{k}g_id - a_id$. Then, for large enough $k$, we have
\begin{equation}\label{eq:tg}
\frac{t_i}{g_i}=d\left(\frac{k+r}{k} - \frac{a_i}{g_i}\right)
\geq d\left(\frac{k+r}{k} - \frac{1}{1 + \frac{r^2}{8k}}\right)
\geq \frac{dr}{k}
\geq \frac{\ln^2 d}{d^{1/2}},
\end{equation}   
where the last inequality follows from $(10k/r)^r\geq d\geq (k/r)^r$ and the range of $r$.~Applying Theorem~\ref{containerlemmalayer} with $\varphi = d/2$ and $\psi=\sqrt{d\ln d}$, we find a family $\mathcal{W}(a_i, g_i)$ such that for every $A_i\in \G(a_i, g_i)$, there exists a pair of sets $(S_i, F_i)\in \mathcal{W}(a_i, g_i)$ with
\begin{equation}\label{SFcer}
S_i\supseteq [A_i],\quad F_i\subseteq N(A_i), \quad |S_i|\leq \frac{k+r}{k} |F_i| + O\left(\frac{t_i\sqrt{\ln d}}{d^{3/2}}\right) \leq 2g_i.
\end{equation}
We first fix the set pair $(S_i, F_i)$, the number of options for it is 
\[
|\mathcal{W}(a_i, g_i)|\leq 2^n \exp\left(O\left(\frac{g_i\ln^2 d}{d^2}\right) +  O\left(\frac{t_i\sqrt{\ln d}}{d^{3/2}}\right)\right)
\leq \exp\left( O\left(\frac{t_i\sqrt{\ln d}}{d^{3/2}}\right)\right),
\]
where the second inequality follows from $g_i\geq d^3$ and (\ref{eq:tg}).
Define $x$ and $y$ real numbers such that
\begin{equation}\label{def:x,y}
g_i = \left(1+\frac{y}{k}\right)a_i, \qquad |S_i| = \left(1 + \frac{x}{k}\right)g_i.
\end{equation}
Note that we may assume
\begin{equation}\label{eq:boundOnx}
-c\frac{t_ik\sqrt{\ln d}}{d^{3/2}g_i}\leq x \leq r + c\frac{t_ik\sqrt{\ln d}}{d^{3/2}g_i},
\end{equation}
for some sufficiently large constant $c$, where the upper bound directly comes from the upper bound of $|S_i|$ on~(\ref{SFcer}) and $|F_i|\leq g_i$.
If the lower bound did not hold, we would have $|S_i|\leq g_i - c\frac{t_i\sqrt{\ln d}}{d^{3/2}}$, and since $A_i\subseteq S_i$, we then easily have
\[
|\G(a_i, g_i)|\leq \sum_{(S_i, F_i)\in \mathcal{W}(a_i, g_i)} \binom{|S_i|}{\leq a_i} \leq \exp\left( O\left(\frac{t_i\sqrt{\ln d}}{d^{3/2}}\right)\right) 2^{g_i - c\frac{t_i\sqrt{\ln d}}{d^{3/2}}}
\leq 2^{g_i\left(1 - \frac{1}{k\sqrt{d}}\right)},
\]
where the last inequality follows from~(\ref{eq:tg}), which would complete the proof.

Now, take a certificate $(S_i,F_i)$ and a set $A^*$ with $|[A^*]|=a_i$ and $|N(A^*)|=g_i$ associated to the certificate. 
For ease of notation, let $s_i:=|S_i|$ and $f_i:=|F_i|$.
For each element in $G^*\setminus F_i$ where $G^*=N(A^*)$, choose if it is included in $G_i$, so the cost of specifying $G_i\cap G^*$ is $2^{g_i-f_i}$. Since $G_i\setminus G^*\subset N(A_i\setminus A^*)$, there exists a set $Y\subseteq A_i\setminus A^*$ that is a minimum cover for $G_i\setminus G^*$ and define $j:=|Y|$. The cover can be chosen from $S_i\setminus A^*$ in at most $\sum_{j=0}^{g_i-f_i}\binom{s_i-a_i}{j}$ ways. With the vertex cover $Y$ and $G_i\cap G^*$ given, we obtain $G_i$ and with it, $[A_i]$. Now we only have to determine which elements of $[A_i]\setminus Y$ belong to $A_i$.  In total, there are at most 
\begin{equation}\label{eq:defN}
N:= 2^{g_i-f_i}\sum_{j=0}^{g_i-f_i}\binom{s_i-a_i}{j}2^{a_i-j}
\end{equation}
ways to determinate $A_i$. We use \eqref{def:x,y} to get the bound

\begin{equation}\label{eq:boundgf}
g_i-f_i=\frac{k}{k+r}s_i-f_i+\frac{r-x}{k+r}g_i\leq \frac{r-x}{k+r}g_i + O\left(\frac{t_i\sqrt{\ln d}}{d^{3/2}}\right),
\end{equation}
where the last inequality follows from \eqref{SFcer}. For the error term, since $r<k$, we have $t_i=\frac{k+r}{k}g_id -a_id \leq 2g_id$, and then
\begin{equation}\label{eq:errorest}
   O\left(\frac{t_i\sqrt{\ln d}}{d^{3/2}}\right) < \frac{g_i}{d^{0.1}}.
\end{equation} 
We can upper bound $N$ by ignoring the last $2^{-j}$ term in \eqref{eq:defN} and then use  \eqref{eq:boundgf}, to get

\begin{equation}\label{eq:Nbound1}
N \leq 2^{\frac{r-x}{k+r}g_i + \frac{g_i}{d^{0.1}} + a_i} \sum_{j=0}^{\frac{r-x}{k+r}g_i + \frac{g_i}{d^{0.1}}}\binom{s_i-a_i}{j}.
\end{equation}
Using~(\ref{eq:errorest}) and \eqref{eq:boundOnx}, we have $x\geq -k/d^{0.1}$. Since \eqref{eq:lowerBoundd} gives $d\geq (k/r)^r$, we have $r>k/d^{0.1}$ for large enough $k$. Using these lower bounds for $x$ and $r$ and substituting $a_{i}=\frac{k}{k+y}g_i$ in \eqref{eq:Nbound1} gives

\begin{equation}\label{eq:Nbound2}N\leq 2^{\frac{r-x}{k+r}g_i + \frac{g_i}{d^{0.1}} + \frac{k}{k+y}g_i} \sum_{j=0}^{\frac{r-x}{k+r}g_i + \frac{g_i}{d^{0.1}}}\binom{s_i-a_i}{j}\leq 2^{\frac{2r}{k}g_i + \frac{k}{k+y}g_i} \sum_{j=0}^{\frac{2r}{k}g_i}\binom{s_i-a_i}{j}. \end{equation}
Note that from \eqref{def:x,y} and \eqref{eq:glower}, it follows that 
\begin{equation}\label{eq:ybound}
    y> \frac{r^2}{10}.
\end{equation} 



From \eqref{def:x,y} it follows that $s_i-a_i = g_i + \frac{x}{k}g_i - \frac{k}{k+y}g_i$.
Now let us look at the binomial sum on the right side of \eqref{eq:Nbound2}. We have $$M:=\sum_{j=0}^{\frac{2r}{k}g_i}\binom{s_i-a_i}{j}\leq \sum_{j=0}^{\frac{2r}{k}g_i} \binom{g_i\left(1+\frac{x}{k} -\frac{k}{k+y}\right)}{j}\leq \sum_{j=0}^{\frac{2r}{k}g_i}\binom{2g_i}{j},$$
where in the last inequality we used that $x<k$ and $y>0$.
Using the known bound $\sum_{i=0}^{q}\binom{m}{i} \le (em/q)^q$ for  $q\leq m/2$, we get

$$M \leq  \left( \frac{2e}{\frac{2r}{k}}\right)^{\frac{2r}{k}g_i} =  2^{\frac{2r}{k}g_i\log(k)}.$$

\noindent Using~(\ref{eq:Nbound2}) then gives

\begin{equation}\label{eq:almostfinalboundN}
\log N \leq \frac{2r}{k}g_i + \frac{g_i}{1+\frac{y}{k}} + \frac{2r\log k}{k}g_i \leq\frac{3r\log k}{k}g_i + \frac{k}{k+y}g_i.
\end{equation}

Recall from~(\ref{eq:ybound}) that $y\geq r^2/10$. Since $r\geq 100\ln k$ and $d\leq (10k/r)^r$, we then have $ky\geq 8(k +y)\log d$, and therefore
\begin{equation}\label{eq:upperboundfracy}
\frac{k}{k+y}\leq 1-\frac{8\log  d}{k}.
\end{equation}
On the other hand, as \eqref{eq:lowerBoundd} gives $d\geq (k/r)^r$, we have  \begin{equation}\label{eq:upperboundfraclog}
    \frac{3r\log k}{k}\leq \frac{7\log d}{k},
\end{equation} 
for large enough $k$. Therefore summing  \eqref{eq:upperboundfracy} and \eqref{eq:upperboundfraclog}, from \eqref{eq:almostfinalboundN} we get 
\[
\log N\leq g_i-g_i\frac{\log d}{k}\leq g_i\left(1-\frac{1}{k\sqrt{d}}\right), 
\]
as desired.

Once we figure out the bound of each individual $|\G(a_i, g_i)|$, we can bound $C_g$ with
\[
\begin{split}
C_g &\leq \sum_{1\leq a\leq g} ~ \sum_{\substack{a=a_s+a_\ell\\ g=g_s+g_\ell}} ~ \Bigg( \sum_{\substack{a_s=\sum a_i,\  g_s=\sum g_i\\ d\leq g_i\leq d^3}}\prod_i |\G(a_i, g_i)|\sum_{\substack{a_\ell=\sum a_i,\  g_\ell=\sum g_i\\ g_i\geq d^3}}\prod_i |\G(a_i, g_i)|\Bigg)\\
&\leq \sum_{1\leq a\leq g} ~ \sum_{\substack{a=a_s+a_\ell\\ g=g_s+g_\ell}} ~ \Bigg(\sum_{\substack{a_s=\sum a_i,\  g_s=\sum g_i\\ d\leq g_i\leq d^3}} 2^{g_s/2}\sum_{\substack{a_\ell=\sum a_i,\  g_\ell=\sum g_i\\ g_i\geq d^3}} 2^{g_\ell\left(1 - \frac{1}{k\sqrt{d}}\right)}\Bigg)\\
&\leq \sum_{1\leq a\leq g} ~ \sum_{\substack{a=a_s+a_\ell\\ g=g_s+g_\ell}}2^{\frac{2g_s\log (ed)}{d}} 2^{g_s/2}\cdot 2^{\frac{2g_{\ell}\log (ed^3)}{d^3}}2^{g_{\ell}\left(1 - \frac{1}{k\sqrt{d}}\right)}\\
&\leq \sum_{1\leq a\leq g} ~ \sum_{\substack{a=a_s+a_\ell\\ g=g_s+g_\ell}}2^{g\left(1 - \frac{1}{2k\sqrt{d}}\right)}
\leq 2^{n}\cdot 2^{2n} \cdot 2^{g\left(1 - \frac{1}{2k\sqrt{d}}\right)}
\leq 2^{g - \frac{\sqrt{d}}{3k}},
\end{split}
\]
where the third inequality follows from Proposition~\ref{prop:decom}, and the last inequality follows from $\frac{g}{k\sqrt{d}}\geq \frac{\sqrt{d}}{k}\gg n$, as $g\geq d$ and $r\gg 1$.
Finally, recall from~(\ref{eq:Cgappears}) that we have
\[
|\mathcal{I}|\leq \sum_{g}C_g 2^{\binom{n-1}{k-1}-g}\leq \sum_{g}2^{\binom{n-1}{k-1}-\frac{\sqrt{d}}{3k}}
\leq 2^{\binom{n-1}{k-1}}\cdot 2^n \cdot2^{-\frac{\sqrt{d}}{3k}}
\leq 2^{\binom{n-1}{k-1}} \cdot2^{-\frac{\sqrt{d}}{4k}},
\]
which completes the proof.

\end{proof}


\section{Proof of Theorem~\ref{containerlemmalayer}}
\label{sec:contlemproof}
We will derive Theorem~\ref{containerlemmalayer} as a corollary of the following theorem.
For $1\leq \varphi\leq s-1$, let
\begin{equation}\label{def:mphi}
m_{\varphi}:=\min\{|N(K)|: y\in Y,\ K\subseteq N(y),\ |K|>\varphi\}.
\end{equation}
\begin{thm}\label{containerlemma}
Let $\Sigma$ be a $(\qq, s)$-biregular graph with bipartition $X\cup Y$. 
For integers $a$, $g$, let
\[
\mathcal{G}(a, g)=\{A\subseteq X: A \text{ is a 2-linked set},\ |[A]|=a,\ |N(A)|=g\},
\]
and set $t:=gs -a\qq$.
Let $1\leq \varphi \leq s-1$, $1\leq \psi \leq \min\{\qq, s\} -1$, 
and $C>0$ be an arbitrary number such that $C\ln \qq/(\varphi \qq)<1$. 
Then there exists a family of containers $\mathcal{W}(a, g)\subseteq 2^{X}\times 2^Y$ of size
\[
\begin{split}
|\mathcal{W}(a, g)|
&\leq |Y|\exp\left(\frac{54Cg\ln \qq\ln(\qq s)}{\varphi \qq} + \frac{54g\ln(\qq s)}{\qq^{Cm_{\varphi}/(\varphi \qq)}}+ \frac{54t\ln s\ln(\qq s)}{\qq(s - \varphi)}\right)\cdot\\
&\quad\binom{\frac{3Cgs\ln \qq}{\varphi \qq}}{\leq \frac{3Ct\ln \qq}{\varphi \qq}}\binom{gs}{\leq t/((s-\varphi)\psi)} \binom{gs\qq}{\leq t/((\qq - \psi)\psi)},
\end{split}
\]
and a function $f: \mathcal{G}(a, g) \rightarrow \mathcal{W}(a, g)$ such that for each $A\in \mathcal{G}(a, g)$, $(S, F):=f(A)$ satisfies:
\begin{itemize}
    \item[(i)] $S\supseteq [A]$, $F\subseteq G$;
    \item[(ii)] $d_F(u)\geq \qq - \psi$ for every $u\in S$;
    \item[(iii)] $d_{X\setminus S}(v)\geq s- \psi$ for every $v\in Y\setminus F$;
    \item[(iv)] $|S| \leq \frac{s}{\qq}|F|  + \frac{\psi t}{\qq}\left(\frac{1}{\qq -\psi} + \frac{1}{s -\psi}\right)$.
\end{itemize}
\end{thm}

Theorem~\ref{containerlemma} is essentially a result of Sapozhenko~\cite{sapozhenko1987number}, which is originally written in Russian and now referred to as the \textit{Sapozhenko's graph container lemma}.
Here we summarize and restate his main result in terms of biregular bipartite graphs.
For the sake of completeness and for providing reference for future work, we give a self-contained proof in the Appendix.
The proof of Theorem~\ref{containerlemma} is very similar to what is in Galvin's expository note~\cite{galvin2019independent} on Sapozhenko's proof for regular bipartite graphs. 
The only difference is that our $t$ is defined to be $gs- aq$, i.e.~the number of edges from $N(A)$ to $X\setminus [A]$, while the Galvin's proof uses $t=g - a$, the difference of the set sizes, as he works on regular graphs.

\begin{proof}[Proof of Theorem~\ref{containerlemmalayer}]
We will apply Theorem~\ref{containerlemma} on the $(\qq, s)$-biregular graph $\mathcal{H}$ with bipartition $\mathcal{L}_{k+r-1}\cup\mathcal{L}_{k-1}$, where $\qq=d$ and $s=(k+r)d/k$.
First, by Theorem~\ref{thm:iso}(ii), we have $m_{\varphi}\geq C_{iso}\varphi d$ for some number $C_{iso}$, and so for large enough $d$ we may take $C=10/C_{iso}$.

By Theorem~\ref{containerlemma}, we have
\[
\begin{split}
|\mathcal{W}(a, g)|
\leq &\ |\mathcal{L}_{k-1}|\exp\left(
O\left(\frac{g\ln^2 d}{\varphi d}\right) + O\left(\frac{g\ln d}{d^{10}}\right)+ O\left(\frac{t\ln^2 d}{d(d - \varphi)}\right)\right)
\exp\left(O\left(\frac{t\ln d}{\varphi d}\ln\frac{gd}{t}\right)\right)\\
&\ \exp\left(O\left(\frac{t\ln (gd^3/t)}{(d - \varphi)\psi}\right)\right)
+ \exp\left(O\left(\frac{t\ln(gd^4/t)}{(d - \psi)\psi}\right)\right).
\end{split}
\]
By Theorem~\ref{thm:iso}(i), as $a\leq \binom{2k+3r/4}{k+r-1}$, we have $g\geq a\left(1 + \frac{r/4 - 2}{k - (r/4-2)}\right)\geq a\left(1 + \frac{r}{5k}\right)$, and then
\begin{equation}
\frac{gd}{t}=\frac{g}{\frac{k+r}{k}g -a} 
\leq \frac{1}{\frac{k+r}{k} - \frac{1}{1 + r/5k}}
=\frac{1}{1 + \frac{r}{k} -\left( 1  -\frac{r}{5k} + o\left(\frac{r}{5k}\right)\right)}
\leq O\left(\frac{k}{r}\right) \leq O\left(d^{1/r}\right).
\end{equation}
Therefore, we can further simplify the upper bound of $|\mathcal{W}(a, g)|$ to the following:
\[
\begin{split}
|\mathcal{W}(a, g)|
\leq &\ |\mathcal{L}_{k-1}|\exp\left(
O\left(\frac{g\ln^2 d}{\varphi d}\right) + O\left(\frac{g\ln d}{d^{10}}\right)+ O\left(\frac{t\ln^2 d}{d(d - \varphi)}\right)\right)
\exp\left(O\left(\frac{t\ln^2 d}{\varphi d}\right)\right)\\
&\ \exp\left(O\left(\frac{t\ln d}{(d - \varphi)\psi}\right)\right)
+ \exp\left(O\left(\frac{t\ln d}{(d - \psi)\psi}\right)\right)\\
=&\ |\mathcal{L}_{k-1}|\exp\left(
O\left(\frac{g\ln^2 d}{\varphi d}\right) 
+ O\left(\frac{t\ln^2 d}{d(d - \varphi)}\right)
+ O\left(\frac{t\ln^2 d}{\varphi d}\right)\right)\\
&\ \exp\left(O\left(\frac{t\ln d}{(d - \varphi)\psi}\right)
+ O\left(\frac{t\ln d}{(d - \psi)\psi}\right)
\right),
\end{split}
\]
which completes the proof.
\end{proof}

\section*{Acknowledgement}
We greatly thank the anonymous referees for their valuable comments and for pointing out an error in the previous version.

\bibliographystyle{abbrv} 
\bibliography{refs} 

\appendix
\section{Graph container lemma for biregular graphs}
We will use the following lemma, a special case of a fundamental result of Lov\'{a}sz~\cite{lovasz1975ratio} and Stein~\cite{stein1974two}. For a bipartite graph $\Sigma$ with bipartition $X\cup Y$, we say that $Y'\subseteq Y$ \textit{cover}s $X$ if each $x\in X$ has a neighbor in $Y'$.
\begin{lemma}\label{lem:cover}
Let $\Sigma$ be a bipartite graph with bipartition $X\cup Y$, where $|N(x)|\geq a$ for each $x\in X$ and $|N(y)|\leq b$ for each $y\in Y$. Then there exists some $Y'\subseteq Y$ that covers $X$ and satisfies
\[
|Y'|\leq \frac{|Y|}{a}\cdot (1 + \ln b).
\]
\end{lemma}

For $\qq, s\geq 2$, let $\Sigma$ be a $(\qq, s)$-regular bipartite graph  with bipartition $X\cup Y$.
From now on, without further specification we always work on the graph $\Sigma$.

\begin{defi}
A \textit{$\varphi$-approximation} for $A\subseteq X$ is a set $F'\subseteq Y$ satisfying 
\begin{equation}
N(A)^{\varphi} \subseteq F' \subseteq N(A) \quad \text{and}\quad N(F') \supseteq [A],
\end{equation}
where
\[
N(A)^{\varphi}:=\{y\in N(A):\ |N_{[A]}(y)|>\varphi\}.
\]
\end{defi}

\begin{lemma}\label{lem:con1}
Let $\Sigma$ be a $(\qq, s)$-regular bipartite graph with bipartition $X\cup Y$.
For integers $a, g$, let 
\[
\mathcal{G}=\mathcal{G}(a, g):=\{A\subseteq X: A \text{ is }2\text{-linked},\ |[A]|=a,\ |N(A)|=g\}.
\]
Let $1\leq \varphi\leq s-1$, and $C$ be a positive number such that $C\ln \qq/(\varphi \qq)<1$. Let $t:=gs -a\qq.$
Then there exists a family $\mathcal{A}_1\subseteq 2^Y$ of size
\begin{equation}\label{certi1}
|\mathcal{A}_1|\leq |Y|\exp\left(\frac{54Cg\ln \qq\ln(\qq s)}{\varphi \qq} + \frac{54g\ln(\qq s)}{\qq^{Cm_{\varphi}/(\varphi \qq)}}+ \frac{54t\ln s\ln(\qq s)}{\qq(s - \varphi)}\right)\binom{\frac{3Cgs\ln \qq}{\varphi \qq}}{\leq \frac{3Ct\ln \qq}{\varphi \qq}},
\end{equation}
and a map $\pi_1: \mathcal{G} \rightarrow \mathcal{A}_1$ for which
$\pi_1(A):=F'$ is a $\varphi$-approximation of $A$.
\end{lemma}

\begin{proof}
Fix an arbitrary set $A\in\mathcal{G}$, and let $p=\frac{C\ln \qq}{\varphi \qq}$.
\begin{claim}
There is a set $T_0\subseteq N(A)$ such that
\begin{equation}\label{eq:T0bound}
|T_0|\leq 3gp,
\end{equation}
\begin{equation}\label{eq:Obound}
e(T_0, X\setminus [A]) \leq 3tp,
\end{equation}
and
\begin{equation}\label{eq:T0'bound}
|N(A)^{\varphi}\setminus N(N_{[A]}(T_0))|\leq 3g\exp(-pm_{\varphi}).
\end{equation}
\end{claim}
\begin{proof}
Construct a random subset $\mathbf{\tilde{T}}\subseteq N(A)$ by taking each $y\in N(A)$  independently with pro\-bability $p$. It is easy to observe that 
\[
\mathbb{E}(|\mathbf{\tilde{T}}|)=gp
 \quad \text{and} \quad
 \mathbb{E}(e(\mathbf{\tilde{T}}, X\setminus [A]))=e(N(A), X\setminus [A])p=(gs - a\qq)p=tp.
\]
By the definition of $m_{\varphi}$ (see~(\ref{def:mphi})), for every $y\in N(A)^{\varphi}$ we have $|N(N_{[A]}(y))|\geq m_{\varphi}$. Therefore, we have
\[
\begin{split}
\mathbb{E}(N(A)^{\varphi}\setminus N(N_{[A]}(\mathbf{\tilde{T}})))
&=\sum_{y\in N(A)^{\varphi}}P(y\notin N(N_{[A]}(\mathbf{\tilde{T}})))
=\sum_{y\in N(A)^{\varphi}}P(\mathbf{\tilde{T}}\cap N(N_{[A]}(y))=\emptyset)\\
&\leq g(1 - p)^{m_{\varphi}}
\leq g\exp(-pm_{\varphi}).
\end{split}
\]
Applying Markov's inequality, we can find a set $T_0\subseteq N(A)$ satisfying the desired conditions.
\end{proof}
\noindent Define
\begin{equation}\label{def:sets}
T'_0:=N(A)^{\varphi}\setminus N(N_{[A]}(T_0)), 
\quad
L:= T'_0\cup N(N_{[A]}(T_0)),
\quad
\Omega:= E(T_0, X\setminus [A]).
\end{equation}
Let $T_1\subseteq N(A)\setminus L$ be the minimal set that covers $[A]\setminus N(L)$ in the graph $\Sigma([A]\setminus N(L),  N(A)\setminus L)$.
Let $F':= L\cup T_1$, by definition, $F'$ is a $\varphi$-approximation of $A$. 

Next, we study the properties of the sets defined in~(\ref{def:sets}).
First note that $L\supseteq N(A)^{\varphi}$, and then we have 
\[
|N(A)\setminus L|(s - \varphi)\leq e(N(A), X\setminus [A])=t.
\]
Together with Lemma~\ref{lem:cover}, we obtain that
\begin{equation}\label{eq:T1bound}
|T_1|\leq \frac{|N(A)\setminus L|}{\qq}(1 + \ln s)
\leq \frac{3t\ln s}{\qq(s - \varphi)}.
\end{equation}
Let $T:=T_0\cup T'_0\cup T_1$. By (\ref{eq:T0bound}), (\ref{eq:T0'bound}) and (\ref{eq:T1bound}), we have 
\begin{equation}\label{eq:Tbound}
|T|\leq 3gp  + 3g\exp(-pm_{\varphi}) + \frac{3t\ln s}{\qq(s - \varphi)}
= \frac{3Cg\ln \qq}{\varphi \qq} + \frac{3g}{\qq^{Cm_{\varphi}/(\varphi \qq)}}+ \frac{3t\ln s}{\qq(s - \varphi)}:=t_{bound}.
\end{equation}
We also have the following claim.
\begin{claim}\label{claim:Tlink}
$T$ is an $8$-linked subset of $Y$.
\end{claim}
\begin{proof}
We start with an easy argument that $[A]$ is 2-linked. First observe that for every two vertices $u, v\in [A]$, there exists two vertices $u', v'\in A$ such that $d(u, u'), d(v, v')\leq 2$. Since $A$ is a 2-linked set, hence $v, v'$ are 2-linked in $A$, and therefore in $[A]$. Thus, $u$ and $v$ are 2-linked in $[A]$.

Next we show that $F'$ is 4-linked. 
Let $u, v$ be two distinct vertices in $F'$.
First, as $F'\subseteq N(A)=N([A])$, there exists two vertices $u', v'\in [A]$ such that $u\sim u'$ and $v\sim v'$.
Moreover, since $[A]$ is a 2-linked set, then $u'$ and $v'$ are 2-linked in $[A]$, that is, there exists a sequence $u'=v'_1, v'_2,\ldots, v'_{\ell-1}, v'_{\ell}=v'$ in $[A]$ such that $d(v'_i, v'_{i+1})\leq 2$ for each $i\in [\ell-1]$.
Recall that $N(F')\supseteq [A]$. Then for each $v'_i$, where $2\leq i\leq \ell-1$, there exists a vertex $v_i\in F'$ such that $v'_i\sim v_i$. Hence, we obtain a sequence $u=v_1, v_2, \ldots, v_{\ell-1}, v_{\ell}=v$ in $F'$ with $d(v_i, v_{i+1})\leq d(v'_i, v'_{i+1}) + 2\leq 4$ for each $i\in [\ell-1]$, which indicates that $u, v$ are 4-linked.

Now we are ready to prove the claim. First observe that $T\subseteq F'$ and for every $v'\in F\setminus T$, there exists a vertex $v\in T$ such that $d(v, v')\leq 2$.
Let $u, v$ be two distinct vertices in $T$. Since $F'$ is 4-linked, then there exists a sequence $u=v'_1, v'_2, \ldots, v'_{\ell-1}, v'_{\ell}=v$ in $F'$ with $d(v'_i, v'_{i+1})\leq 4$ for each $i\in [\ell-1]$.
Let $v_i$ be the vertex in $T$ with $d(v_i, v'_i)\leq 2$. Then we obtain a sequence $u=v_1, v_2, \ldots, v_{\ell-1}, v_{\ell}=v$ in $T$ with $d(v_i, v_{i+1})\leq d(v_i, v'_i) + d(v'_i, v'_{i+1}) + d(v'_{i+1}, v_{i+1})\leq 2+4+2=8$ for each $i\in [\ell-1]$, which indicates that $u, v$ are 8-linked.
\end{proof}

Observe that $T_0$ and $\Omega$ together determine $N(N_{[A]}(T_0))$. Therefore the set $F'$ is uniquely determined by the set tuple $(T_0, T'_0, T_1, \Omega)$. 
Let $\mathcal{A}_1$ be the collection of all sets $F'$, which are produced in such a way from some set $A\in \mathcal{G}$. Hence, to get an upper bound on the size of $\mathcal{A}_1$, it is sufficient to consider the number of choices for such set tuples.

By Corollary~\ref{cor:numlinkset} and (\ref{eq:Tbound}), the number of choices for $T$ is at most 
\[
|Y|\cdot\sum_{\ell \leq t_{\text{bound}}}\exp(8\ell\ln(\qq s))\leq |Y|\exp(16t_{\text{bound}}\ln(\qq s)).
\]
For a fixed set $T$, the number of choices for $T_0$ and $T'_0$ are both at most $2^{t_{\text{bound}}}$, and then $T_1$ is uniquely determined. Moreover, by~(\ref{eq:Obound}), for a fixed $T_0$, the number of choices for $\Omega$ is at most $\binom{3gps}{\leq 3tp}$. To summarize, we obtain that
\[
\begin{split}
|\mathcal{A}_1|
&\leq |Y|\exp(16t_{\text{bound}}\ln(\qq s))\cdot 2^{t_{\text{bound}}}\cdot 2^{t_{\text{bound}}}\binom{3gps}{\leq 3tp}
\leq |Y|\exp(18t_{\text{bound}}\ln(\qq s))\binom{3gps }{\leq 3tp}\\
&\leq |Y|\exp\left(\frac{54Cg\ln \qq\ln(\qq s)}{\varphi \qq} + \frac{54g\ln(\qq s)}{\qq^{Cm_{\varphi}/(\varphi \qq)}}+ \frac{54t\ln s\ln(\qq s)}{\qq(s - \varphi)}\right)\binom{\frac{3Cgs\ln \qq}{\varphi \qq}}{\leq \frac{3Ct\ln \qq}{\varphi \qq}}.
\end{split}
\]
\end{proof}

\begin{defi}\label{def:approx}
A \textit{$\psi$-approximation} for $A\subseteq X$ is a pair $(S, F)\in 2^X\times 2^Y$ satisfying 
\begin{equation}
F\subseteq N(A), \quad S\supseteq [A],
\end{equation}
\begin{equation}\label{def:psi1}
d_F(u) \geq \qq-\psi \quad \text{ for }\quad u\in S,
\end{equation}
and
\begin{equation}\label{def:psi2}
d_{X\setminus S}(v) \geq s-\psi \quad \text{ for }\quad v\in Y\setminus F.
\end{equation}
\end{defi}

\begin{prop}\label{lem:psiappro}
Let $A$ be a 2-linked subset of $X$ with $|[A]|=a$, $|N(A)|=g$. If $(S, F)$ is a $\psi$-approximation for $A$, then
\begin{equation}\label{eq:psiappr}
|S| \leq \frac{s}{\qq}|F|  + \frac{\psi(gs- a\qq)}{\qq}\left(\frac{1}{\qq -\psi} + \frac{1}{s -\psi}\right).
\end{equation}
\end{prop}
\begin{proof}
First, by (\ref{def:psi1}) and (\ref{def:psi2}) we have
\[
\qq|S|  - \psi |S\setminus [A]|=
\qq|[A]|  + (\qq - \psi) |S\setminus [A]|\leq 
e(S, N(A)) \leq s|F| +\psi|N(A)\setminus F|,
\]
which gives
\begin{equation}\label{eq:Sbound}
|S| \leq \frac{s}{\qq}|F| +  \frac{\psi}{\qq} \Big(|N(A)\setminus F| + |S\setminus [A]|\Big).
\end{equation}
Note that each $v\in N(A)\setminus F$ contributes at least $s-\psi$ edges to $E(N(A), X\setminus [A])$, while each $v\in S\setminus [A]$ contributes at least $\qq -\psi$ edges to $E(N(A), X\setminus [A])$. This implies that
\[
|N(A)\setminus F| + |S\setminus [A]| 
\leq e(N(A), X\setminus [A])\left(\frac{1}{\qq -\psi} + \frac{1}{s -\psi}\right)
= (gs- a\qq)\left(\frac{1}{\qq -\psi} + \frac{1}{s -\psi}\right),
\]
which together with~(\ref{eq:Sbound}) completes the proof.
\end{proof}

\begin{lemma}\label{lem:con2}
Let $\Sigma$, $\mathcal{G}$, $\mathcal{A}_1$ and $t$ be as in Lemma~\ref{lem:con1}.
For $1\leq \varphi \leq s-1$ and a set $F'\in \mathcal{A}_1$, let
\[
\mathcal{G}'=\mathcal{G}'(F'):=\{A\in \mathcal{G}: \ F' \text{ is a $\varphi$-approximation of $A$}\}.
\]
Then for every $1 \leq \psi \leq \min\{\qq, s\} -1$ there exists a family $\mathcal{A}_2\subseteq 2^{X}\times 2^{Y}$ of size
\begin{equation}
|\mathcal{A}_2|\leq \binom{gs}{\leq t/((s-\varphi)\psi)} \binom{gs\qq}{\leq t/((\qq - \psi)\psi)},
\end{equation}
and a map $\pi_2: \mathcal{G}' \rightarrow \mathcal{A}_2$ for which
$\pi_2(A):=(S, F)$ is a $\psi$-approximation of $A$.
\end{lemma}

\begin{proof}
Fix a set $A\in\mathcal{G'}$. We will construct a $\psi$-approximation $(S, F)$ for $A$ via the following two-step algorithm.

\begin{itemize}
\item[\textbf{Step 1.}] We start the algorithm with $F_1=F'$ and an empty set $P_1$. In the $i$-th iteration, we pick a
vertex $v_i\in [A]$ with $d_{N(A)\setminus F_1}(v_i)>\psi$. In case there are multiple choices, we give
preference to vertices that come earlier in some arbitrary predefined ordering of $X$. Then we update $F_1$ by $F_1\cup N(v_i)$ and $P_1$ by $P_1\cup\{v_i\}$, and move to the next iteration. We stop the process when $\{v\in[A]: d_{N(A)\setminus F_1}(v)>\psi\}=\emptyset$. Let $S_1=\{v\in X: d_{F_1}(v)\geq \qq - \psi\}$ and move to Step 2.

\item[\textbf{Step 2.}] We start with $S_2=S_1$, and an empty set $P_2$.
In the $j$-th iteration, we pick a
vertex $u_j\in Y\setminus N(A)$ with $d_{S_2}(v_i)>\psi$ (we break the ties similarly as before). Then we update $S_2$ by $S_2\setminus N(u_j)$ and $P_2$ by $P_2\cup\{u_j\}$, and move to the next iteration. We stop the process when $\{u\in Y\setminus N(A): d_{S_2}(u)>\psi\}=\emptyset$. In the end, we let $S=S_2$, $F_2=\{u\in Y: d_S(u)>\psi\}$, and $F=F_1\cup F_2$.
\end{itemize}

Next we verify that such a pair $(S, F)$ is indeed a $\psi$-approximation of $A$. Recall that $F'\subseteq N(A)$, as $F'$ is a $\varphi$-approximation of $A$. Then the procedures in Step 1 immediately shows that $F_1\subseteq N(A)$. Observe that $F_2\subseteq N(A)$, as otherwise Step 2 would not stop. Therefore, we have $F\subseteq N(A)$. Similarly, we observe that $S_1\supseteq [A]$, as otherwise Step 1 would not stop. Since in Step 2 only neighbors of $Y\setminus N(A)$ were deleted from $S_1$, we still have $S\supseteq [A]$.
Conditions (\ref{def:psi1}) and (\ref{def:psi2}) immediately follow from the definitions of $S_1, F_2$, and $S\subseteq S_1$, $F\supseteq F_1, F_2$.

Note that the output of the algorithm is completely determined by the sets $P_1$ and $P_2$. Initially. since $F'\supseteq N(A)^{\varphi}$, we have $|N(A)\setminus F'|\cdot(s - \varphi)\leq e(N(A), X\setminus [A])=t$, which gives $|N(A)\setminus F'|\leq t/(s-\varphi)$. 
Each iteration in Step 1 removes at least $\psi$ vertices from $N(A)\setminus F_1$, so there are at most $t/((s-\varphi)\psi)$ iterations. Therefore, we have
\[|P_1|\leq t/((s-\varphi)\psi),\]
and each $v\in P_1$ is an element in $[A]$ and hence $N(F')$, a set of size at most $gs$.

Similarly, for Step 2 initially we have $|S_1\setminus [A]|\leq t/(\qq - \psi)$, as each vertex in $S_1\setminus [A]$ contributes at least $\qq - \psi$ edges to $E(N(A), X\setminus [A])$.
Each iteration removes at least $\psi$ vertices from the set $S_2\setminus [A]$ and so there are at most $t/((\qq - \psi)\psi)$ iterations. Therefore, we have
\[|P_2|\leq t/((\qq - \psi)\psi),\]
and each $v\in P_2$ is an element in $N(S_1)\subseteq N(N(F_1))$, a set of size at most $gs\qq$.

Let $\mathcal{A}_2$ be the collections of all pairs $(S, F)$ which can be produced from the above algorithm by some set $A\in \mathcal{G}'$. From the above discussion, we have
\[
|\mathcal{A}_2|\leq \binom{|N(F')|}{\leq t/((s-\varphi)\psi)}\binom{|N(N(F_1))|}{t/((\qq - \psi)\psi)}
\leq \binom{gs}{\leq t/((s-\varphi)\psi)} \binom{gs\qq}{t/((\qq - \psi)\psi)}.
\]
\end{proof}

Theorem~\ref{containerlemma} immediately follows from Lemmas~\ref{lem:con1},~\ref{lem:con2}, and Proposition~\ref{lem:psiappro}.
\begin{proof}[Proof of Theorem~\ref{containerlemma}]
By Lemmas~\ref{lem:con1},~\ref{lem:con2}, for each set $A\in \G(a, g)$, there exists a set pair $(S, F)\in 2^X \times 2^Y$ such that $(S, F)$ is a $\psi$-approximation of $A$.
Let $\mathcal{W}(a, g)$ be the collection of all such $(S, F)$ pairs, and then we have 
\[
\begin{split}
|\mathcal{W}(a, g)|\leq |\A_1||\A_2|
&\leq |Y|\exp\left(\frac{54Cg\ln \qq\ln(\qq s)}{\varphi \qq} + \frac{54g\ln(\qq s)}{\qq^{Cm_{\varphi}/(\varphi \qq)}}+ \frac{54t\ln s\ln(\qq s)}{\qq(s - \varphi)}\right)\cdot\\
&\quad\binom{\frac{3Cgs\ln \qq}{\varphi \qq}}{\leq \frac{3Ct\ln \qq}{\varphi \qq}}\binom{gs}{\leq t/((s-\varphi)\psi)} \binom{gs\qq}{\leq t/((\qq - \psi)\psi)}.
\end{split}
\]
Finally, by Definition~\ref{def:approx} and Proposition~\ref{lem:psiappro}, such a $(S, F)$ pair satisfies conditions (i)--(iv).
\end{proof}

\end{document}